\documentclass[11pt,a4paper]{article}
\usepackage[OT1]{fontenc}
\usepackage[english]{babel}
\usepackage{amsmath}
\usepackage{amsfonts}
\usepackage{amsthm}
\def\RR{\mathbb{R}}
\def\CC{\mathbb{C}}

\def\NN{\mathbb{N}}
\def\n{ {\cal N} }

\def\h{ {\cal H} }
\def\ele{ {\cal L} }
\def\b{ {\cal B} }

\def\s{ {\cal S} }

\def\gg{ {\mathbb G} }
\def\noi{\noindent}

\def\rank{ \mathrm{rank} }
\def\g1{ \mathfrak{g}_1  } 
\def\gp{ \mathfrak{g}_p  }
\def\gq{ \mathfrak{g}_q   }
\def\vp{  \varphi }

\newtheorem{teo}{Theorem}[section]
\newtheorem{prop}[teo]{Proposition}
\newtheorem{lem}[teo]{Lemma}
\newtheorem{coro}[teo]{Corollary}

\theoremstyle{definition}
\newtheorem{rem}[teo]{Remark}
\newtheorem{ejem}[teo]{Example}
\newtheorem{nota}[teo]{Notation}

\title{The compatible Grassmannian}
\author{E. Andruchow, E. Chiumiento, M. E. Di Iorio y Lucero}

\begin{document}

\maketitle 

\begin{abstract}
Let $A$ be a positive injective operator in a Hilbert space $(\h,<\ ,\ >)$, and denote by $[\ , \ ]$ the  inner product defined by $A$: $[f,g]=<Af,g>$. A closed subspace $\s \subset \h$ is called  $A$-compatible if there exists a closed complement for $\s$, which is orthogonal to $\s$ with respect to the inner product $[\ ,\ ]$. Equivalently, if there exists a 
necessarily unique idempotent operator $Q_\s$ such that $R(Q_\s)=\s$, which is symmetric for this inner product. The compatible Grassmannian $Gr_A$ is the set of all $A$-compatible subspaces of $\h$. By parametrizing it via the one to one correspondence $\s\leftrightarrow Q_\s$, this set is shown to be a differentiable submanifold of the Banach space of all operators in $\h$ which are symmetric with respect to the form $[\ , \ ]$. A Banach-Lie group acts naturally on the compatible Grassmannian, the group of all invertible operators in $\h$ which preserve the form $[\ , \ ]$.  Each connected component in $Gr_A$ of a compatible subspace $\s$ of finite dimension, turns out to be a symplectic leaf in a Banach Lie-Poisson space.
For $1\le p \le \infty$, in the presence of a fixed  $[\ , \ ]$-orthogonal decomposition of $\h$, $\h=\s_0\dot{+} \n_0$, we study the restricted compatible Grassmannian (an analogue of the restricted, or Sato Grassmannian). This restricted compatible Grassmannian is  shown to be a submanifold of the Banach space of $p$-Schatten operators which are symmetric for the form $[\ , \ ]$. It carries the locally transitive action of the Banach-Lie group of invertible operators which preserve $[\ , \ ]$, and are of the form $G=1+K$, with $K$ in $p$-Schatten class. The connected components of this restricted Grassmannian are characterized by means of of the Fredholm index of pairs of projections. Finsler metrics which are isometric for the group actions are introduced for both compatible Grassmannians, and minimality results for curves are proved. 
\end{abstract}
\bigskip



\section{Introduction}
Let $0< A\le 1$ be a positive  operator with trivial kernel in a Hilbert space $(\h,<\ ,\ >)$. Then $A$ defines an inner product $[\ ,\ ]=[\ ,\ ]_A$ on $\h$ by means of 
$$
[f,g]=<Af,g>, \ f,g\in\h.
$$
Apparently, the form $[\ ,\ ]$ is continuous with respect to the topology  of $\h$. If $A$ is not invertible, $\h$ is not complete with this inner product. We shall denote by $\ele$ the completion of $\h$ with this inner product. A closed linear subspace $\s\subset \h$ is called {\it compatible with $A$} ({\it or $A$-compatible}) if it admits a supplement which is orthogonal with respect to the inner product defined by $A$. In this paper, we study the {\it compatible Grassmannian}, namely
\[  Gr_A= \{ \,  \s \subset \h \, : \, \text{$\s$ is compatible  with $A$}   \,  \}.   \]
Since  $A$ has trivial kernel, if $\s$ is compatible with $A$, then the supplement is unique, and it is given by $A(\s)^\perp$.
This allows us to identify each compatible subspace $\s$ with the projection $Q_\s$ with range $\s$ and kernel $A(\s)^\perp$. Thus the 
compatible Grassmannian may be regarded as the following set
\[ Gr_A=\{ \, Q \in \b(\h) \,  : \, Q^2=Q, \, Q^*A=AQ  \,  \}, \]
where we will consider the uniform topology inherited from $\b(\h)$. The proof of the afore-mentioned facts and examples of compatible and non compatible subspaces can be found in \cite{cms,cms02}, whereas for additional information on applications  
 of compatible subspace to statistics, sampling, signal processing  and frames the reader  should see the references in  \cite{cms05}.
 
Note that there are two notions of orthogonality, the one given by the usual inner product $<\ ,\ >$ of $\h$, and the one given by $[\ ,\ ]$, which we will call $A$-orthogonality. This 
leads us to consider, as  useful tools in this work, the results on operators on spaces with two norms by M. G. Krein \cite{krein}, J. Dieudonn\'e \cite{dieudonne}, P. D. Lax \cite{lax}, I. C. Gohberg and M. K. Zambicki\v{\i} \cite{gz}. 
  An operator will be called {\it $A$-symmetric} (resp. {\it $A$-anti-symmetric}, {\it $A$-unitary}) if it is symmetric (resp. anti-symmetric, unitary) with respect to the $A$ inner product $[\ ,\ ]$. It is apparent that
$B$ is $A$-symmetric (resp. $A$-anti-symmetric) if and only if  
$
B^*A=AB 
$ (resp. $B^*A=-AB$). 
It also follows that  $A$-unitary operators form a subgroup $\gg_A$ of the linear invertible group $Gl(\h)$, which can be characterized as  
\[  \gg_A=\{  \,  G \in Gl(\h) \, : \, G^*AG=A   \,  \}. \] 
Moreover, it is a Banach-Lie group endowed with the uniform topology of $\b(\h)$ and its Lie algebra can be identified 
with the $A$-anti-symmetric operators (see \cite{achl}). However,  $\gg_A$ is not necessarily a complemented  
submanifold of $\b(\h)$ if $A$ is not invertible. We prove a result 
stating that the Banach-Lie algebra of $\gg_A$ is  complemented in $\b(\h)$ if and only if the kernel of a given nilpotent derivation is 
complemented  (Theorem \ref{equivalence compl}). 

If $\s$ is compatible with $A$ and $G\in\gg_A$, then $G(\s)$ is compatible with $A$ and $Q_{G(\s)}=GQ_\s G^{-1}$. Thus $\gg_A$ acts on $Gr_A$ by similarity.



Our main results concern the differentiable structure of $Gr_A$. We show that $Gr_A$ is a complemented submanifold of the real Banach space of $A$-symmetric operators, and that the action of $\gg_A$ is locally 
transitive and  has smooth local cross sections (Corollary \ref{est difer}). The compatible Grassmannian has a remarkable property with respect to the space ${\cal Q}(\h)$ of all idempotents in $\b(\h)$, studied by G. Corach, H. Porta and L. Recht in \cite{cpr}. In this paper a linear connection was introduced in ${\cal Q}(\h)$, and its geodesics were computed: these are curves of the form $\delta (t)=e^{tX}Qe^{-tX}$ for specific operators $X$. We show here that if two elements $Q_1$ and $Q_2$ of $Gr_A$ lie at norm distance less than $r=r_{Q_1}$, then the unique geodesic of ${\cal Q}(\h)$ joining them lies inside $Gr_A$ (see Proposition \ref{estimate r}; Corollary \ref{geodesicas}).

In the case where $Q \in Gr_A$ has finite rank, and $\h$, $\ele$ are special function spaces,   
the orbit of $Q$ is  the Grassmann manifold arising in the   Hartree-Fock theory (see \cite{chm}). We show that these orbits are strong  symplectic leaves in a Banach Lie-Poisson space (Corollary \ref{symplectic leaves}). These  infinite dimensional  Poisson structures were introduced by   A. Odzijewicz and T. Ratiu in  \cite{or}.

We also define a {\it restricted compatible Grassmannian} $Gr_{res, p}$, analogous to the restricted Grassmannian (also called Sato Grassmannian, \cite{pressleysegal,w}). Given $1\le p\le \infty$ and a fixed $A$-orthogonal decomposition
$$
\h=\s_0 \, \dot{+} \, \n_0 ,
$$
with $Q_0=Q_{\s_0}$ (so that $\n_0=N(Q_0)$, the kernel of $Q_0$), an element $Q$ of $Gr_A$ belongs to $G_{res, p}$ if
$$
Q-Q_0\in\b_p(\h),
$$
where $\b_p(\h)$ denotes the class of  $p$-Schatten operators in $\h$. We show that the Banach-Lie group 
$$
\gg_{p,A}=\{G\in \gg_A: G-1\in\b_p(\h)\}
$$
acts on $Gr_{res, p}$, that the action is locally transitive, and that the orbits are the connected components of $Gr_{res, p}$. Moreover, these orbits are characterized by the Fredholm index of pairs of projections defined by J. Avron, R. Seiler and  B. Simon in \cite{ass}. Given a pair $(P_1,P_2)$ of orthogonal projections, the index of the pair is the index of the operator
$$
P_2P_1|_{R(P_1)}\to R(P_2),
$$
if this operator is Fredholm (in which case the pair $(P_1,P_2)$ is called a Fredholm pair, otherwise the index is infinite). If $Q_1,Q_2\in Gr_{res, p}$, then their extensions $\bar{Q}_1, \bar{Q}_2$ are orthogonal projections on $\ele$, and form a Fredholm pair. In Theorem \ref{components gres} we show that they lie in the same connected component of $Gr_{res, p}$ (or equivalently, are conjugate by the action of $\gg_{p,A}$) if and only if the index of the pair is zero (or equivalently, $index(\bar{Q}_1,\bar{Q}_0)=index(\bar{Q}_2,\bar{Q}_0)$).
These components are  differentiable manifolds in the local structure given by the $p$-norm. Also in the restricted compatible Grassmannian, it is shown that geodesics of ${\cal Q}(\h)$ joining sufficiently close elements of $Gr_{res, p}$, lie inside $Gr_{res, p}$.

We introduce non complete Finsler metrics  for both $Gr_A$ and $Gr_{res, p}$, in terms, respectively, 
of the operator norm and the $p$-norm in $\ele$. Being isometric for the respective group actions, these metrics 
are a natural choice. We  prove that the geodesics described above joining sufficiently close elements  have minimal length (see Proposition \ref{minimal}; Corollary \ref{minimalp}).

Let us give a brief outline of the contents  of this paper. Preliminary results on operators on spaces 
with two norms are stated in Section \ref{on preliminary facts}, as well as the analysis of when the real space of $A$-anti-symmetric operators is complemented in $\b(\h)$. We conjecture that this happens if and only if $A$ is invertible. In Section \ref{eamples of compatibility} we give examples of compatible subspaces. In Section \ref{diff structure gra} we examine the local structure of $Gr_A$, and the properties of the action of $\gg_A$. The property of permanence of geodesics joining close elements is proved, and an estimate of this geodesic radius is given. In Section \ref{Banach Lie P} we introduce a Banach-Lie-Poisson structure  in the finite rank components of $Gr_A$. In Section \ref{res comp gra} we introduce the restricted compatible Grassmannian $Gr_{res, p}$. We prove the  local regular structure, the characterization of the components in terms of the index of pairs, and the permanence of geodesics. In Section \ref{minimal curves} we introduce invariant Finsler metrics in $Gr_A$ and $Gr_{res, p}$, and we prove the minimality results.

\begin{nota}
The Hilbert space  $\h$ is endowed with two (inner product) norms, $\| \, \cdot \, \|$ will be the usual norm and $\| \, \cdot \, \|_A$ the one given by the $A$-inner product. Clearly, $\|f\|_A\le \|f\|$ for any $f\in\h$ (since $A\le 1$). 
In addition, 
$\b^A_s(\h)$ and $\b^A_{as}(\h)$ will denote the Banach spaces of (respectively) $A$-symmetric and $A$-anti-symmetric operators, regarded as subspaces of $\b(\h)$. The operator norm on $\b(\h)$ will be denoted by $\| \, \cdot \, \|$, meanwhile $\| \, \cdot \, \|_{\b(\ele)}$ will be the operator norm of $\b(\ele)$.
\end{nota}

\section{Preliminaries}\label{on preliminary facts}

Note that any $G\in\gg_A$ can be extended to a unitary operator $\bar{G}$ acting on $\ele$. 
In \cite{achl} it was shown that $G\in \gg_A$ if and only if $G=U|_{\h}$, where $U$ is a unitary operator on $\ele$ such that $U$ leaves the dense subspace $\h$ fixed, i.e. $U(\h)=\h$.

We shall denote by $\sigma_\h(T)$ the spectrum of $T$ as an operator in $\h$, and by $\sigma_\ele(\bar{T})$ the spectrum of its extension (when it exists) to $\ele$.  Let us recall the following  facts, 
adapted from their original broader context to our case:

\begin{teo}[M. G. Krein \cite{krein}, P. D. Lax \cite{lax}]\label{symmetrizable results}
Let $B$ be an $A$-symmetric operator. The following assertions hold:
\begin{enumerate}
\item[i)] $\bar{B}$ is bounded in $\ele$.
\item[ii)] $\sigma_\ele(\bar{B}) \subseteq \sigma_\h(B)$.
\item[iii)] If $\lambda$ belongs to the point spectrum of $B$ as an operator on $\h$, then $\lambda$ belongs to the point spectrum of $\bar{B}$ as an operator on $\ele$. Moreover, the $\lambda$-eigenspace   over $\h$ and $\ele$ is the same.
\item[iv)] If $B$ is a compact operator on $\h$, then $\bar{B}$  is a compact operator on $\ele$ and $\sigma_\ele(\bar{B})=\sigma_\h(B)$.
\end{enumerate}
\end{teo}
\begin{rem}\label{ext map bounded}
It is not difficult to see that if $B$ is $A$-symmetric, then
$ 
\|\bar{B}\|_{{\b}(\ele)} \leq \| B\|_{{\b}(\h)} \,  
$. Also note that $A$ itself is $A$-symmetric, and that its extension $\bar{A}$ remains positive definite.
\end{rem}

Operators which are $A$-symmetric are often called {\it symmetrizable}.

A generalization of the above results can be found in \cite{gz}. A  bounded operator acting  on $\h$ is called {\it proper} if it  has a bounded adjoint with respect to the inner product on $\ele$. This means that $B \in \b(\h)$ is proper if and only if  for every $f \in \h$, there is a vector $g \in \h$ satisfying $[Bf,h]=[h,g]$ for all $h \in \h$. This allows to define $B^+ f=g$. Actually, $B^+$ is the restriction to $\h$   of the $\ele$-adjoint of $B$. In particular, 
 $B$ is $A$-symmetric if and only if $B=B^+$.

\begin{teo}[I. C. Gohberg, M. K. Zambicki\v{\i} \cite{gz}]\label{proper op}
Let $B$ be a proper operator. The following assertions hold:
\begin{enumerate}
	\item[i)] $\bar{B}$ is bounded as an operator on $\ele$.
	\item[ii)] $\sigma_\ele(\bar{B}) \subseteq \sigma_\h(B) \cup \overline{\sigma_\h(B^+)}$, where the bar indicates complex conjugation.

  \item[iii)] If $B$ is a compact operator on $\h$, then $\bar{B}$  is a compact operator on $\ele$. Moreover, $\sigma_\ele(\bar{B})=\sigma_\h(B)$ and the eigenspaces   in  $\ele$ and $\h$ corresponding to the non zero eigenvalues  coincide.
\end{enumerate}
\end{teo}
 
\noi The following result will also be useful.

\begin{teo}[J. Dieudonn\'e \cite{dieudonne}]\label{Quasiherm}
Let $B$ be an $A$-symmetric operator. Then there is a unique symmetric operator $X$ in $\h$ such that $A^{1/2}B=XA^{1/2}$.
\end{teo}
We finish this section on preliminaries by giving a characterization of the case when $\b_s^A(\h)$ is a complemented (real) subspace of $\b(\h)$. Note that $\b_{as}^A(\h)=i \b_s^A(\h)$, and therefore the first subspace is complemented if and only if the second is, and $\mathbb{S}$ is a 
supplement for $\b_s^A(\h)$ if and only if $i\mathbb{S}$ is a supplement for $\b_{as}^A(\h)$.

Consider the Hilbert space $\h\times\h$ and the operator $A_0$ on $\h\times\h$ given by
$$
A_0=\left( \begin{array}{ll} 0 & A \\ 0 & 0 \end{array} \right).
$$
Note that $A_0^2=0$. Recall that $A$ is  injective. Then a straightforward computation shows that the  operator $B\in\b(\h\times\h)$ commutes with $A_0$ if and only if
$$
B=\left( \begin{array}{ll} X & Y \\ 0 & Z \end{array} \right),
$$
with $XA=AZ$.
\begin{teo}\label{equivalence compl}
The real subspace $\b_s^A(\h)$ is complemented in $\b(\h)$ if and only if the commutant of $A_0$ is complemented (as a complex subspace) in $\b(\h\times \h)$.
\end{teo}
\begin{proof}
By the form of the commutant of $A_0$, namely
$$
\left( \begin{array}{ll} 0 & Y \\ 0 & 0 \end{array} \right)\oplus \left( \begin{array}{ll} X & 0 \\ 0 & Z \end{array} \right),
$$
with $XA=AZ$, it is apparent that it is complemented if and only if the complex subspace
$$
{\cal D}=\{(X,Z)\in \b(\h)\times\b(\h): XA=AZ\}
$$ is complemented in $\b(\h)\times\b(\h)$. Note that ${\cal D}$ decomposes as a direct sum by means of
$$
(X,Z)=\frac12 (X+Z^*, X^*+Z)+\frac12 (X-Z^*,Z-X^*),
$$
where $X^*+Z\in \b_s^A(\h)$ and $Z-X^*\in \b_{as}^A(\h)$. Indeed,
$$
(X^*+Z)^*A=XA+Z^*A=AZ+(AZ)^*=AZ+(XA)^*=A(Z+X^*),
$$
and similarly for $Z-X^*$. Thus 
${\cal D}={\cal D}_s\oplus {\cal D}_{as}$, where
$$
{\cal D}_s=\{(V^*,V): V\in\b_s^A(\h)\}
$$
and 
$$
{\cal D}_{as}=\{(W^*,-W): W\in\b_{as}^A(\h)\}.
$$
It is apparent that ${\cal D}_s\cap {\cal D}_{as}=\{0\}$.

Suppose first that $\b_s^A(\h)\oplus \mathbb{S}=\b(\h)$ for some closed real subspace $\mathbb{S}\subset \b(\h)$, so that also $\b_{as}^A(\h)\oplus i\mathbb{S}=\b(\h)$. Then
$$
\mathbb{T}=\{(R^*,R)+(-T^*,T): R\in\mathbb{S}, T\in i\mathbb{S}\}
$$
is a supplement for ${\cal D}$ (note that the sum $(R^*,R)+(-T^*,T)$ is direct), and moreover, it is a complex subspace of $\b(\h)\times \b(\h)$. Indeed, if $(R^*,R)+(-T^*,T)\in\mathbb{T}$, then 
$$
i((R^*,R)+(-T^*,T))=((iT)^*,iT)+((iR)^*,-iR)\in \mathbb{T},
$$
because $iT\in \mathbb{S}$ and $iR\in i\mathbb{S}$.

Conversely, suppose that ${\cal D}$ is complemented in $\b(\h)\times\b(\h)$,
$$
\b(\h)\times\b(\h)=\mathbb{T}\oplus {\cal D}=\mathbb{T}\oplus {\cal D}_{as}\oplus {\cal D}_s.
$$
Then the real subspace ${\cal D}_s=\{  (V^*,V)  :  V\in\b_s^A(\h)  \}$ is complemented in $\b(\h)\times\b(\h)$.
Let $E:\b(\h)\times\b(\h)\to {\cal D}_s\subset \b(\h)\times\b(\h)$ be an idempotent projecting onto ${\cal D}_s$. Put $$\
\mathbb{D}=\{  (T^*,T)  : T\in\b(\h) \}.
$$
Then ${\cal D}_s\subset \mathbb{D}$ and 
$$
E|_{\mathbb{D}}:\mathbb{D}\to {\cal D}_s\subset \mathbb{D}
$$
is an idempotent operator, and the set of pairs $(V^*,V)$, $V\in\b_s^A(\h)$ are complemented in the set of pairs $(T^*,T)$, $T\in\b(\h)$. It follows that $\b_s^A(\h)$ is complemented in $\b(\h)$.
\end{proof}

Apparently, if $A$ is invertible, the kernel of the nilpotent derivation on $\b(\h \times \h)$ given by $\delta_{A_0}(X)=XA_0-A_0X$ is complemented. We conjecture that this is also a necessary condition. Derivations with closed (and complemented) range have been characterized. However, to our knowledge, there are no results characterizing derivations with complemented kernel.

\section{Examples of compatible subspaces}\label{eamples of compatibility}
First note that there exist non-compatible subspaces if $A$ is non invertible \cite{cms}. 
It is not difficult to see that if $\s$ is finite dimensional, then it is $A$-compatible. The same holds for finite co-dimensional subspaces. If $P$ is an orthogonal projection which commutes with $A$, then $P$ is $A$-symmetric (for instance, if $P$ is a spectral projection of $A$). Therefore $R(P)$ is compatible. We shall call these special type of compatible subspaces (ranges of orthogonal projection commuting with $A$), {\it commuting subspaces}.

In \cite{achl} several examples of $A$-unitary operators where considered. One way to obtain compatible subspaces is by taking $\s=G(\s_0)$, where $\s_0$ is commuting and $G\in \gg_A$.
\begin{ejem}
Let $\h=H^1_0(0,1)$, the subspace of the Sobolev space obtained as the closure with respect to the inner product
\begin{equation}\label{inner prod sob}
<f,g>=\int_0^1 f(t)\bar{g}(t) dt+\int_0^1 f'(t)\bar{g'}(t) dt
\end{equation}
of smooth functions on $(0,1)$ with compact support. The second inner product $[\ , \ ]$ is the usual $L^2(0,1)$ inner product, so that $\ele=L^2(0,1)$. In this case the operator $A$ which implements $[\ , \ ]$ is the solution operator of the Sturm-Liouville problem
$$
\left\{ \begin{array}{l} u-u''=f \\ u(0)=u(1)=0, \end{array} \right. 
$$
that is, $Af=u$, the unique solution $u$ of the equation above for a given $f\in \h$. In this case $A$ is compact, and it is diagonalized with the eigenvectors $s_k(t)=\frac{\sqrt2}{\sqrt{k^2\pi^2+1}}\sin(k\pi t)$ with eigenvalues $\lambda_k=(k^2\pi^2+1)^{-1}$. Thus $A$ is an infinite diagonal matrix with different positive numbers in the diagonal. Therefore any operator commuting with $A$ is also diagonal. In particular, any commuting subspace $\s_0$ is generated by a collection of eigenvectors $s_k$. Pick any infinite proper collection of $s_k$, containing $s_1$ and $s_2$. Consider the operator $G=M_{e^{i\pi t}}|_{H^1_0(0,1)}$, i.e. $(Gf)(t)=e^{i\pi t}f(t)$. Clearly, $M_{e^{i\pi t}}$ is a unitary operator on $L^2(0,1)$, and also it is apparent that 
$$
M_{e^{i\pi t}}(H^1_0(0,1))\subset H^1_0(0,1)
$$
and 
$$
M_{e^{i\pi t}}^{-1}(H^1_0(0,1))=M_{e^{-i\pi t}}(H^1_0(0,1))\subset H^1_0(0,1),
$$
i.e., by the above remark, $G\in \gg_A$. An elementary computation shows that $<G(s_1),s_j>\ne 0$  for all $j$ odd, and $<G(s_2),s_k>\ne 0$ for all $k$ even. It follows that $G(\s_0)$ is a compatible subspace, which is non  commuting.
\end{ejem}
\begin{ejem}\label{ej sobolev}
Consider $\h=H^1(\mathbb{R})\subset L^2(\mathbb{R})$, or rather, its Fourier transform 
$$\h=\{\, f\in L^2(\mathbb{R}) \,: \, (1+x^2)^{1/2}f(x)\in L^2(\mathbb{R}) \, \},$$ with the complete inner product
$$
<f,g>=\int_{\mathbb{R}} (1+x^2)f(x)\bar{g}(x) d x.
$$
The second inner product is the usual $L^2$-inner product, and the operator $A$ implementing it in $\h$ is 
$Af(x)=\frac{1}{1+x^2}f(x)$. Consider the reflection $Rf(x)=f(-x)$. A simple computation shows that $R$ is symmetric on $\h$, and that it commutes with $A$. Then the closed subspace $\s_0=\{f\in\h: Rf=f\}$ of a.e. even functions, is a commuting subspace. For $a\in\mathbb{R}$, consider the translation operator $T_af(x)=f(x-a)$. Note that $T_a$ is a unitary operator on $L^2$, which leaves $\h$ fixed. Therefore $T_a\in \gg_A$. Note also that 
$$
T_aRT_{-a}f(x)=f(-x+2a),
$$
so that  $\s_a=T_a(\s_0)=\{f\in\h: f(x)=f(-x+2a) \hbox{ a.e.}\}$ is a compatible subspace. Another simple computation shows that $T_aRT_{-a}$ is not symmetric if $a\ne 0$. The subspace $\s_a$  is non commuting if $a\ne 0$. Indeed,
$$
Q_{\s_a}=T_aQ_{\s_0}T_{-a}=\frac12 (1+T_aRT_{-a})
$$
is not symmetric.
\end{ejem} 
Both examples are constructed as translations of a commuting subspace. A natural question would be, if every compatible subpace $\s$ is of the form $\s=G(\s_0)$, where $\s_0$ is a commuting subspace and $G\in\gg_A$. 
The following simple example shows that in general this is not the case.
\begin{ejem}
Let $\h=L^2(0,1)$ with the usual inner product, and $$\ele=\{f\in\ele:  f(t)t^{1/2}\in  L^2(0,1)\}$$ given by $A=M_t$, $(Af)(t)=tf(t)$, i.e.
$$
[f,g]=\int_0^1 tf(t)\overline{g(t)} dt.
$$
The operator $A$ generates a maximal abelian von Neumann algebra in $\b(\h)$ (namely $L^\infty(0,1)$ acting as multiplication operators). It follows that any non trivial  projection commuting with $A$ is of the form $P=M_{\chi_\Delta}$, where $\chi_\Delta$ is the characteristic function of a measurable set $\Delta$ of positive Lebesgue measure. It follows that any 
commuting subspace is infinite and co-infinite dimensional. Pick any finite dimensional subspace $\s$. Then $\s$ is compatible, but it is not of the form $G(\s_0)$ for any commuting subspace $\s_0$.
\end{ejem}

\begin{rem}\label{rem connected}
In \cite{chm} it was shown that for any pair of finite dimensional subspaces $\s_1, \s_2$ of the same dimension, there exists an element $G\in\gg_A$ such that $G(\s_1)=\s_2$. Moreover,  it was remarked that the element $G$ constructed is of the form $1+$ finite rank. In \cite{achl} it was shown that such elements $G$ are exponentials: there exists a finite rank element in $Z\in \b_{as}^A(\h)$ such that $G=e^Z$. Then the curve $e^{tZ}Q_{\s_1}e^{-tZ}$ joins $Q_{S_1}$ and $Q_{S_2}$. Thus  the subspaces $\s_1$ and $\s_2$ lie in the same connected component of $Gr_A$.
\end{rem}

There is an alternative characterization for compatible subspaces in terms of extensions to the Hilbert space $\ele$:

\begin{prop} \label{caract comp}
Let $\s$ be a closed subspace of $\h$. Denote by $\overline{\s}$ its closure in $\ele$. Then $\s$ is $A$-compatible if and only if $\overline{\s}\cap \h=\s$ and the orthogonal projection $P_{\overline{\s}} \in \b(\ele)$ satisfies $P_{\overline{\s}}(\h)\subseteq \h$.
\end{prop}
\begin{proof}
Suppose that $\s\subset \h$ is compatible, then there exists $Q_\s$, the unique idempotent with range $\s$ which is $A$-symmetric. Note that  $Q_\s$ extends to an orthogonal projection $\bar{Q}_\s$ on $\ele$, with $\s \subset R(\bar{Q}_\s)\subset \overline{\s}$, i.e. $P_{\overline{\s}}=\bar{Q}_\s$. Thus 
$$
P_{\overline{\s}}(\h)=Q_\s(\h)=\s\subset \h.
$$
Clearly, $\s \subset \overline{\s}\cap \h$. If $f\in \overline{\s}\cap \h$, then $f=P_{\overline{\s}}(f)=Q_\s(f)\in\s$.

Conversely, suppose that $\s= \overline{\s}\cap \h$ and that $P_{\overline{\s}}(\h)\subset \h$. Then $P_{\overline{\s}}|_{\h}$ is an $A$-symmetric idempotent in $\h$. Its range is $P_{\overline{\s}}(\h)=\overline{\s}\cap\h=\s$.
\end{proof}

The following example shows that the hypothesis   $\overline{\s}\cap \h=\s$ in the above proposition does not follow from the assumption $P_{\overline{\s}}(\h)\subseteq \h$.

\begin{ejem} \label{no equiv assump}
Let $H^1(0,1)$ be the space of functions $f \in L^2(0,1)$  that have  weak derivative $f' \in L^2(0,1)$. It is a Hilbert space with the norm given by equation (\ref{inner prod sob}) and $H_0^1(0,1)$ is a proper closed subspace of $H^1(0,1)$ (see e.g. \cite{treves}). Set $\s=H_0^1(0,1)$, $\h=H^1(0,1)$ and $\ele=L^2(0,1)$. Since $H_0 ^1(0,1)$ is a dense subspace of $L^2(0,1)$, then $\overline{\s}\cap \h= \ele \cap \h=\h\neq  \s$. However, the orthogonal projection onto $\overline{\s}=\ele$, which is the identity map, trivially leaves $\h$ invariant.   
\end{ejem}
Let $B$ be an $A$-symmetric operator. Pick $f \in \overline{\ker(B)} \cap \h$. Then there is a sequence $(f_n)_n $ such that 
$Bf_n=0$ and $f_n \to f$ in the norm of $\ele$. Since $B$ is $\ele$-continuous by  Theorem \ref{symmetrizable results}, one obtains 
that $Bf=\lim B f_n=0$, and consequently, $f \in \ker(B)$. Hence $\overline{\ker(B)} \cap \h=\ker(B)$. Despite of being satisfied one of the assumptions
in Proposition \ref{caract comp}, the following example shows that the kernel of an $A$-symmetric operator is not a compatible subspace in general.

\begin{ejem}
Consider again the Sobolev space $\h=H^1(0,1)$ as in Example \ref{no equiv assump}. Take a smooth 
function $\theta: (0,1) \to \RR$ which is 
equal to zero   in a proper subinterval $(0,a]$ of $(0,1)$ and positive in the complement $(a,1)$. Recall that $H^1(0,1)$ is a dense subspace of $\ele=L^2(0,1)$, and each function in $H^1(0,1)$ admits an absolutely continuous representative. Consider the multiplication operator
\[  \bar{M}_{\theta}: L^2(0,1) \to L^2(0,1) , \, \, \, \, \, \, \bar{M}_{\theta} f = \theta f . \]
Since $\theta$ is a real-valued function, this operator is self-adjoint. Moreover, $\bar{M}_{\theta}$ leaves $H^1(0,1)$ invariant because 
the function $\theta$ is smooth. Therefore $\bar{M}_{\theta}$  is the extension of an $A$-symmetric operator $M_{\theta}$ 
acting on  $H^1(0,1)$. Clearly, the kernel of $M_{\theta}$
 is given by
\[  \ker(M_{\theta})= \{ \,  f \in H^1(0,1)  \, : \,  f\equiv 0  \text{ on } [a,1)   \,  \}.   \]
On the other hand, $\ker(\bar{M}_{\theta})$  is orthogonal to $R(\bar{M}_{\theta})$ due to the fact that $\bar{M}_{\theta}$ is 
self-adjoint. 
The constant function $f_0 \equiv 1$ belongs to $H^1(0,1)$, and  can be written as the sum of two characteristic functions
\[  f_0= \chi_{(0,a]} + \chi_{(a,1)} \, .   \]
 Clearly  $ \chi_{(0,a]} \in \overline{\ker(M_{\theta})}=\ker(\bar{M}_{\theta})$ and  this sum is   orthogonal  with respect the inner product of $L^2(0,1)$. Therefore $P_{\overline{\ker(M_{\theta})}} (f_0) =\chi_{(0,a]} \notin H^1(0,1)$, and consequently, $\ker(M_{\theta})$ is not a compatible subspace.   
\end{ejem}

\section{Differentiable structure of $Gr_A$}\label{diff structure gra}

Let us transcribe the following result contained in the appendix of the paper
\cite{rae}   by I. Raeburn, which is a consequence of the implicit function
theorem in Banach spaces.

\begin{lem}\label{raeburn}
Let $G$ be a Banach-Lie group acting smoothly on a Banach space $X$. For a fixed
$x_0\in X$, denote by $\pi_{x_0}:G\to X$ the smooth map $\pi_{x_0}(g)=g\cdot
x_0$. Suppose that
\begin{enumerate}
\item
$\pi_{x_0}$ is an open mapping,  regarded as a map from $G$ onto the orbit
$\{g\cdot x_0: g\in G\}$ of $x_0$ (with the relative topology of $X$).
\item
The differential $d(\pi_{x_0})_1:(TG)_1\to X$ splits: its kernel and range are
closed complemented subspaces.
\end{enumerate}
Then the orbit $\{g\cdot x_0: g\in G\}$ is a smooth submanifold of  $X$, and the
map
$$
\pi_{x_0}:G\to \{g\cdot x_0: g\in G\}
$$ 
is a smooth submersion.
\end{lem}

We shall apply this lemma to our situation. Fix $Q_0\in Gr_A$, and consider the map
$$
\pi_{Q_0}:\gg_A \to \b^A_s(\h) , \ \pi_{Q_0}(G)=GQ_0G^{-1}.
$$
Clearly, it is a $C^\infty$ map, its differential at the identity $1$ is
$$
d(\pi_{Q_0})_1=\delta_{Q_0}: \b^A_{as}(\h)\to \b_s^A(\h), \ \delta_{Q_0}(X)=XQ_0-Q_0X,
$$
where the Banach-Lie algebra of $\gg_A$ is identified with $\b_{as}^A(\h)$, the space of $A$-anti-symmetric operators.
\begin{prop}\label{delta splits}
The range and the kernel of $\delta_{Q_0}$ are complemented.
The range of $\delta_{Q_0}$ consists of all $A$-symmetric operators that are co-diagonal with respect to $Q_0$, i.e.
$$
R(\delta_{Q_0})=\{Y\in\b^A_s(\h): Q_0YQ_0=(1-Q_0)Y(1-Q_0)=0\}.
$$

The kernel of $\delta_{Q_0}$ consists of all $A$-anti-symmetric operators that are diagonal with respect to $Q_0$, i.e.
$$
\ker(\delta_{Q_0})=\{X\in\b^A_{as}(\h): Q_0X=XQ_0\}.
$$
\end{prop}
\begin{proof} 
Let us prove first the assertion on  $R(\delta_{Q_0})$. Clearly, if $Y=XQ_0-Q_0X$ for some operator $X\in\b^A_{as}(\h)$, then $Y$ is $A$-symmetric, and $Q_0YQ_0=0=(1-Q_0)Y(1-Q_0)$. Conversely, let $Y$ be  an $A$-symmetric operator which is co-diagonal with respect to $Q_0$. Then 
$$
Y=Q_0YQ_0+Q_0Y(1-Q_0)+(1-Q_0)YQ_0+(1-Q_0)Y(1-Q_0)
$$
$$
=Q_0Y(1-Q_0)+(1-Q_0)YQ_0=Q_0Y+YQ_0.
$$
Let $X=YQ_0-Q_0Y$. Clearly $X$ is $A$-anti-symmetric (being the commutator of two $A$-symmetric operators). Moreover, note that
$$
\delta_{Q_0}(X)=(YQ_0-Q_0Y)Q_0-Q_0(YQ_0-Q_0Y)=YQ_0 + Q_0Y=Y.
$$
The assertion on the kernel of $\delta_{Q_0}$ is trivial. Let us prove that these spaces are complemented. To this end, consider the linear map
$$
E=E_{Q_0}:\b(\h)\to \b(\h), \ E(X)=Q_0XQ_0+(1-Q_0)X(1-Q_0).
$$
Clearly, $E$ is idempotent. Note that $E(\b^A_s(\h))\subset \b^A_s(\h)$ and $E(\b^A_{as}(\h))\subset \b^A_{as}(\h)$. Indeed, if $X\in \b^A_s(\h)$, then
$$
E(X)^*A=Q_0^*X^*Q_0^*+(1-Q_0^*)X^*(1-Q_0^*)A=Q_0^*X^*AQ_0+(1-Q_0^*)X^*A(1-Q_0)
$$
$$
=Q_0^*AXQ_0+(1-Q_0^*)AX(1-Q_0)=AQ_0XQ_0+A(1-Q_0)X(1-Q_0)=AE(X),
$$
and similarly if $X\in \b^A_{as}(\h)$.
Let $E_s$ and  $E_{as}$ be the idempotents obtained by restricting $E$ to (respectively) $\b^A_s(\h)$ and $\b^A_{as}(\h)$. The range of $E_{as}$ consists of $A$-anti-symmetric operators which commute with $Q_0$, i.e. $R(E_{as})=\ker(\delta_{Q_0})$, and thus it is complemented in $\b^A_{as}(\h)$. The kernel of $E_s$ consists of $A$-symmetric operators which are co-diagonal with respect to $Q_0$, thus $\ker(E_s)=R(\delta_{Q_0})$, and therefore this space is complemented in $\b^A_s(\h)$.
\end{proof}
The next result implies in particular that the map $\pi_{Q_0}$ is open, when it is regarded as a map from $\gg_A$ onto the orbit $\{GQ_0G^{-1}: G\in\gg_A\}$. This result is based on general facts of the space of idempotents on $\b(\h)$. The main ideas are contained in \cite{cpr}. We state them in the following remark.
\begin{rem}
The set ${\cal Q}(\h)$ of idempotents is a complemented analytic submanifold of $\b(\h)$. Its tangent space at a given element $Q\in {\cal Q}(\h)$ is the Banach space $\{XQ-QX: X\in\b(\h)\}$. There is a natural linear connection, induced by the action of the general linear group, and the direct decomposition of $\b(\h)$ at every $Q\in {\cal Q}(\h)$, as $Q$-diagonal operators plus $Q$-co-diagonal operators. Geodesics of this connection, starting at the point $Q$, are of the form
$$
\delta(t)=e^{tX}Qe^{-tX},
$$
where $X$ is a $Q$-co-diagonal element in $\b(\h)$. 

The affine correspondence $Q\leftrightarrow \epsilon_Q=2Q-1$ allows to view idempotents alternatively as symmetries, i.e. operators $\epsilon$ such that $\epsilon^2=1$. 
Some calculations are easier with symmetries than with idempotents. For instance, $X$ is co-diagonal with respect to $Q$ if and only if it anti-commutes with $\epsilon_Q$. In particular, the above geodesic, regarded as a curve of symmetries, is
$$
\epsilon_\delta(t)=e^{tX}\epsilon_Qe^{-tX}=e^{2tX}\epsilon_Q=\epsilon_Qe^{-2tX}.
$$
Finally, another feature of the diagonal, co-diagonal decomposition (called grading in \cite{cpr}) is that it enables one to produce local cross sections of the action of the general linear group, by means of the exponential map. Namely,
the differential at $1$ of the $C^\infty$ map
$$
\{X\in\b(\h): QXQ=(1-Q)X(1-Q)=0\} \to {\cal Q}(\h), \ \ X\mapsto e^{X}Qe^{-X}
$$
is the linear isomorphism
$$
\{X\in\b(\h): QXQ=(1-Q)X(1-Q)=0\}\to \{XQ-QX: X\in\b(\h)\}, \ \ X\mapsto XQ-QX.
$$
Indeed, it is a simple computation to prove that this linear map is its own inverse. Thus, by the inverse mapping theorem in Banach spaces, the former map is a local diffeomorphism. There exists an open set ${\cal V}_Q\subset {\cal Q}(\h)$, $Q\in {\cal V}_Q$, such that for any $R\in{\cal V}_Q$ there is a unique $Q$-co-diagonal $X_R$ such that 
$e^{X_R}Qe^{-X_R}=R$. The map $R\to X_R$ is analytic, and therefore an analytic cross section is defined:
$$
{\cal V}_Q \ni R \mapsto e^{X_R} \in Gl(\h).
$$
Let us call it the exponential cross section. Note that the open set ${\cal V}_Q$ can be chosen to be a ball in $\b(\h)$, intersected with $\mathcal{Q}(\h)$: ${\cal V}_Q=\{R\in {\cal Q}(\h): \|R-Q\|<r_Q\}$. Moreover, we can further shrink $r_Q$ so that if $R\in{\cal V}_Q$, $\|e^{2X_R}-1\|<1$. We fix this value of $r_Q$ for each $Q$ (we will estimate this radius below).
\end{rem}
\begin{prop}\label{cont cross}
For each $Q_0\in Gr_A$, the exponential cross section just described, when restricted to ${\cal V}_{Q_0}\cap Gr_A$, takes values in $\gg_A$. Thus  the map $\pi_{Q_0}:\gg_A\to \{GQ_0G^{-1}: G\in\gg_A\}$ has continuous local cross sections, and in particular it is an  open map.
\end{prop}
\begin{proof}
Let $Q\in {\cal V}_{Q_0}\cap Gr_A$. Then $Q=e^{X_Q}Q_0e^{-X_Q}$, and thus we get
$$
\epsilon_Q=e^{2X_Q}\epsilon_{Q_0}=\epsilon_{Q_0}e^{-2X_Q}.
$$
Then $e^{2X_Q}=\epsilon_{Q}\epsilon_{Q_0}$ and $\epsilon_{Q_0}\epsilon_{Q}=e^{-2X_Q}$, so we have that 
$$
(e^{2X_Q})^*A=(2Q_0^*-1)(2Q^*-1)A=A(2Q_0-1)(2Q-1)=Ae^{-2X_Q}.
$$
Hence $e^{2X_Q}\in \gg_A$. Note that by the assumption on $r_{Q_0}$, it follows that $\|e^{2X_Q}-1\|<1$. In \cite{achl}, it was shown that if $G\in\gg_A$ and $\|G-1\|< 1$, then the usual series of the logarithm of $G$ (which is absolutely convergent) converges to an element which belongs to $\b^A_{as}(\h)$. It follows that $X_Q\in \b^A_{as}(\h)$. Hence $e^{X_Q}\in\gg_A$.
\end{proof}
Note that what in fact was proved above, is that the smooth map 
$$
{\cal V}_{Q_0}\cap Gr_A \ni Q\mapsto X_Q
$$ 
takes values in $\b_{as}^A(\h)$,  the Banach-Lie algebra of $\gg_A$.
We may use Lemma \ref{raeburn}  to prove the following: 

\begin{coro}\label{est difer}
The orbit $\gg_A\cdot Q_0=\{GQ_0G^{-1}: G\in\gg_A\}$  is a complemented $C^\infty$ submanifold of $\b_s^A(\h)$. The map $\pi_{Q_0}:\gg_A\to \gg_A\cdot Q_0$ is a $C^\infty$ submersion. Moreover, the orbit $\gg_A\cdot Q_0$ is a union of connected components of $Gr_A$, and therefore, $Gr_A$ is a complemented submanifold of $\b_s^A(\h)$.
\end{coro}
\begin{proof}
The assertion that the orbit $\gg_A\cdot Q_0$ is a union of connected components of $Gr_A$ needs proof. By Proposition \ref{cont cross}, which provides the existence of continuous local cross sections for the action of $\gg_A$, it follows that  this action is locally transitive, a fact which implies the assertion. Also it implies that an element of $Q\in Gr_A$, which verifies $\|Q-Q_0\|<r_{Q_0}$, can be connected to $Q_0$ by the curve $e^{tX_{Q_0}}Q_0e^{-tX_{Q_0}}$ in $Gr_A$.
\end{proof}
\begin{rem}
In the case where $\rank(Q_0)<\infty$, it can be shown that $\gg_A\cdot Q_0$ is also a  complemented submanifold of $\b(\h)$ (see \cite{chm}). 
\end{rem}

Note that the curve in the above proof is a geodesic of ${\cal Q}(\h)$. Therefore we have also the following consequence:
\begin{coro}\label{geodesicas}
Let $Q,Q_0\in Gr_A$ such that $\|Q-Q_0\|<r_{Q_0}$. Then there is a unique geodesic $\delta$ (up to reparametrization) of ${\cal Q}(\h)$ which joins them, and $\delta(t)\in Gr_A$ for all $t\in\mathbb{R}$.
\end{coro}
We may compute an estimate for $r_{Q_0}$. As remarked in the Section \ref{on preliminary facts}, if $B,C\in\b_s^A(\h)$, then they have symmetric extensions $\bar{B},\bar{C}$ in $\ele$, and $\|\bar{B}-\bar{C}\|_{\b(\ele)}\le \|B-C\|$. It follows that if $Q,Q_0\in Gr_A$ satisfy  $\|Q-Q_0\|<1$, then $\bar{Q}, \bar{Q_0}$ 
are self-adjoint projections in $\ele$, lying at distance less than $1$. In \cite{pr} it was shown that two 
self-adjoint projections at distance less than $1$ are joined by a unique geodesic of the manifold of self-adjoint projections. Thus there exists a unique  anti-symmetric operator $Z_{\bar{Q}}$ acting on $\ele$, which is co-diagonal with respect to $\bar{Q}_0$, such that 
$$
e^{2Z_{\bar{Q}}}\epsilon_{\bar{Q}_0}=\epsilon_{\bar{Q}_0}e^{-2Z_{\bar{Q}}}=\epsilon_{\bar{Q}}.
$$
Note that $e^{2Z_{\bar{Q}}}=\epsilon_{\bar{Q}}\epsilon_{\bar{Q}_0}$. By construction, $\epsilon_{\bar{Q}}$ and $\epsilon_{\bar{Q}_0}$ leave $\h\subset \ele$ fixed. Therefore $e^{2Z_{\bar{Q}}}(\h)\subset \h$. Similarly,
its inverse $e^{-2Z_{\bar{Q}}}=\epsilon_{\bar{Q}_0}\epsilon_{\bar{Q}}$ leaves $\h$ invariant. It follows that
$e^{2Z_{\bar{Q}}}$ is a unitary operator in $\ele$ such that $e^{2Z_{\bar{Q}}}(\h)=\h$. Therefore 
$$
e^{2Z_{\bar{Q}}}|_{\h}\in \gg_A.
$$
Let us further shrink the distance between $Q_0$ and $Q$ so that the exponent $Z_{\bar{Q}}$ induces an operator in the Banach-Lie algebra of $\gg_A$. To this effect, it will suffice that 
$$
\|e^{2Z_{\bar{Q}}}|_{\h}-1\|<1 \hbox{ and } \|e^{-2Z_{\bar{Q}}}|_{\h}-1\|<1.
$$
Indeed, on one hand, since $e^{2Z_{\bar{Q}}}|_{\h}-1$  is extendable to $\ele$, by Theorem \ref{proper op},
$$
\sigma_\ele(e^{2Z_{\bar{Q}}}-1)\subset \sigma_\h(e^{2Z_{\bar{Q}}}|_{\h}-1)\cup \overline{\sigma_\h(\left(e^{2Z_{\bar{Q}}}|_{\h}\right)^+-1)}.
$$
Note that 
$$
\left(e^{2Z_{\bar{Q}}}|_{\h}\right)^+=e^{-2Z_{\bar{Q}}}|_{\h}=\epsilon_{\bar{Q}_0}\epsilon_{\bar{Q}}|_\h=\epsilon_{Q_0}\epsilon_{Q}.
$$
If $\rho_\h$ and $\rho_\ele$ denote the spectral radii, then
$$
\|e^{2Z_{\bar{Q}}}-1\|_{\b(\ele)}=\rho_\ele(e^{2Z_{\bar{Q}}-1})\le \max\{ \rho_\h(e^{2Z_{\bar{Q}}}|_{\h}-1), \rho_\h(e^{-2Z_{\bar{Q}}}|_{\h}-1)\} 
$$
$$
\le \max\{ \|e^{2Z_{\bar{Q}}}|_{\h}-1\|, \|e^{-2Z_{\bar{Q}}}|_{\h}-1\| \}<1.
$$
This implies in particular that $2Z_{\bar{Q}}$ is the usual determination of the logarithm (note that $\|Z_{\bar{Q}}\|<\pi/6$).
On the other hand, by  \cite[Lemma 3.3]{achl},  $\|e^{2Z_{\bar{Q}}}|_{\h}-1\|<1$ implies that the usual logarithm series for $e^{2Z_{\bar{Q}}}|_{\h}$ converges in $\b(\h)$ to an element in the the Banach-Lie algebra of $\gg_A$. Therefore it  follows that $X_{Q}:=Z_{\bar{Q}}|_{\h}\in \b_{as}^A(\h)$, as wanted. Note that
$e^{2Z_{\bar{Q}}}|_{\h}=\epsilon_{Q}\epsilon_{Q_0}$, and then
\begin{align*}
\|e^{2Z_{\bar{Q}}}|_{\h}-1\|& =\|\epsilon_{Q}\epsilon_{Q_0} -1\|=\|(\epsilon_{Q}-\epsilon_{Q_0})\epsilon_{Q_0}\|\le \|\epsilon_{Q}-\epsilon_{Q_0}\|\|\epsilon_{Q_0}\|
 =2\|\epsilon_{Q_0}\|\|Q-Q_0\|.
\end{align*}
Similarly, 
$$
\|e^{-2Z_{\bar{Q}}}|_{\h}-1\|=\|\epsilon_{Q_0}\epsilon_{Q}-1\|=\|\epsilon_{Q_0}(\epsilon_Q-\epsilon_{Q_0})\|\le 2\|\epsilon_{Q_0}\|\|Q-Q_0\|.
$$
We may summarize the above discussion in the following 
\begin{prop}\label{estimate r}
With the current notations,  $r_{Q_0}=\frac{1}{2\|\epsilon_{Q_0}\|}=\frac{1}{2} \frac{1}{\|Q_0\|+(\|Q_0\|^2-1)^{1/2}}$.
\end{prop}
\begin{proof} The norm of $\epsilon_{Q_0}$ was computed in \cite{gerisch}:
$$
\|\epsilon_{Q_0}\|=\|Q_0\|+(\|Q_0\|^2-1)^{1/2}.
$$\end{proof}

\section{Finite rank orbits as symplectic leaves}\label{Banach Lie P}

The foundations of Banach Poisson differential geometry were investigated by A. Odzijewicz and  T. Ratiu in \cite{or}. One of the key features of this new theory is that  Banach Poisson manifolds provide an appropriate  setting for an unified approach 
to the Hamiltonian and the quantum mechanical description of  physical systems. In addition to the seminal article \cite{or}, we also refer the reader to the  monograph \cite{od} for a detailed exposition on Banach Poisson manifolds; meanwhile several interesting examples and applications can be found in \cite{br,brt,go}. 

An important class of infinite dimensional linear Poisson manifolds is given by Banach Lie-Poisson spaces: a Banach space $\mathfrak{b}$ is called a \textit{Banach Lie-Poisson space} if  $\mathfrak{b}^*$ is a Banach Lie algebra endowed with a Lie bracket $[ \, \cdot \, , \, \cdot \, ]$ such that $\text{ad}_x ^*(\mathfrak{b}) \subseteq \mathfrak{b}$ for all $x \in \mathfrak{b}^*$. Recall that $\text{ad}_x ^*: \mathfrak{b}^{**} \to \mathfrak{b}^{**}$ is the coadjoint representation, i.e. the dual map to the adjoint representation defined by $\text{ad}_x:\mathfrak{b}^* \to \mathfrak{b}^*$, $\text{ad}_x (y)=[x,y]$. Actually, this is an equivalent characterization of Banach Lie-Poisson spaces (see \cite[Theorem 4.2]{or}), whereas the original definition involves the notion of Banach Poisson manifolds that we have omitted here. 
Notable examples of Banach Lie-Poisson spaces are the dual of any reflexive Banach Lie algebra, the space of trace class operators on a Hilbert space, and more generally, the predual of a von Neumann algebra.

Let $\b_p(\h)$ denote the class of $p$-Schatten  operators on $\h$ ($1\leq p \leq \infty$), and let $Tr$ be the usual trace on the Hilbert space $\h$. 
 Recall that  the $p$-norm of an operator in $X \in \b_p(\h)$ is given by $\|X\|_p=Tr(|X|^p)^{1/p}$. 
  If $p=\infty$, $\b_{\infty}(\h)$ denotes the compact operators on $\h$. Denote by $\gg_{p,A}$ the  group consisting of operators in $\gg_A$ which are $\b_p(\h)$-perturbations of the identity:
$$
\gg_{p,A}=\{G\in\gg_A: G-1\in\b_p(\h)\}.
$$
The group $\gg_{p,A}$ is a connected Banach-Lie group with the topology defined by the metric $(G_1 , G_2)\mapsto \| G_1 - G_2 \|_p$. Indeed, in \cite{achl}  it was proved that $\gg_{p,A}$ is an exponential group. This means that $\gg_{p,A}=exp(\mathfrak{g}_p)$, where $\mathfrak{g}_p$ is its  Lie algebra, i.e. 
\[  \mathfrak{g}_p =  \b_{as}^A(\h) \cap \b_p(\h).    \]

We will prove that the Lie algebra $\mathfrak{g}_p$ of the Banach Lie group $\gg_{A,p}$ is a Banach Lie-Poisson space. To this end, we will need
the following proposition,  which  shows that the usual duality relationships are preserved in the class of $A$-anti-symmetric operators.
   
\begin{prop}\label{duality}
The map
\begin{equation}\label{eq duality}
\gp \to (\gq )^* , \, \, \, \, Z \mapsto Tr(Z \, \cdot \,)
\end{equation}  
is an isometric isomorphism of real Banach spaces if $1 < p < \infty$ and $p^{-1} + q^{-1}=1$. Moreover, the same map implements the real isometric isomorphisms  $(\mathfrak{g}_{\infty})^* \simeq \mathfrak{g}_1$ and $(\mathfrak{g}_1)^* \simeq \b^A_{as}(\h)$.
\end{prop}
\begin{proof}
We will  only prove the case where $1<p<\infty$; the other two cases are analogous. We first check that the isomorphism 
is well defined. Clearly, any functional of the form $X \mapsto Tr(ZX)$, $Z \in \gp$\,, $X \in \gq$, is continuous. 
To show that this functional is real  note that $ZX$ is a proper  operator. According to Theorem \ref{proper op}, and noting that $ZX$ is a  compact operator on $\h$,  it follows that 
$\sigma_{\h}(ZX)=\sigma_{\ele}(\bar{Z}\bar{X})$, taking into account the multiplicity of the nonzero eigenvalues. Moreover, $ZX$ is a trace class operator on $\h$ by H\"older's inequality  and its extension $\bar{Z}\bar{X}$ is also trace class on $\ele$ by Lalesco's inequality \cite{lalesco}. Then by Lidskii's theorem  we find that
\[  Tr(ZX)=\sum_{i=1}^{\infty} \lambda_i(ZX) = Tr_{\ele}(\bar{Z}\bar{X}),  \]
where  $\lambda_i(ZX)$ are the eigenvalues counted with multiplicity and  $Tr_{\ele}$ is the usual trace in $\b(\ele)$. Since the operators $Z,X$ can be extended to anti-symmetric compact operators on $\ele$ 
and  the trace of the product of two anti-symmetric operators is real, we get that $Tr(ZX)$ is real. 

 Let us prove that the map in equation 
 (\ref{eq duality}) is surjective. Let $\vp$ be a functional in $(\gq )^*$. Extend this functional to the  complex vector space $S:=\CC \, \gq$ by defining $\vp_1(wX)=w\vp(X)$, $w \in \CC$. It is well defined: if $zX=wY$, then $-\bar{z}X=(zX)^+=(wY)^+=-\bar{w}Y$. Thus, $Re(z)X=Re(w)Y$ and  $Im(z)X=Im(w)Y$. Suppose that $Re(z)\neq 0$, then 
$$\vp_1(zX)=z\vp(X)=z\vp (\tfrac{Re(w)}{Re(z)}Y)=z \tfrac{Re(w)}{Re(z)} \vp(Y)=w \vp(Y)=\vp_1(wY) $$   
 Also note that 
\[ \|\vp \|= \sup_{\|X\|_q=1 \, , \, X \in \gq} | \vp(X)| \leq \sup_{\|X\|_q=1 \, , \, X \in S} | \vp_1(X)| = \| \vp_1 \|.   \]    
Pick $X \in \gq$ and $w \in \CC$ such that $\| w X \|_q=1$.  Then,
\[ |\vp_1(wX)| = |w||\vp(X)| = | \vp (\tfrac{X}{\,\|X\|_q})| \leq \| \vp\|.  \]
Hence we have $\| \vp_1 \|=\| \vp\|$. Next $\vp_1$ can be extended  to a functional $\tilde{\vp}_1$ 
on  $\b_q(\h)$ with the same norm, by the Hahn-Banach theorem. Since $(\b_q(\h))^* \simeq \b_p(\h)$, it follows that $\tilde{\vp}_1=Tr(Z \, \cdot \,)$, for some $Z \in \b_p(\h)$.

Let us prove that $Z$ is an $A$-anti-symmetric operator. 
To this end, we consider the following rank one proper operators: $(f\otimes g)(h)=[h,g]f$ for any $f,g,h \in \h$.
It is easily seen that $(f\otimes g)^+=g\otimes f$. Set $X=f\otimes g$. Then these operators may be decomposed 
as the sum of  real and imaginary parts with respect to the adjoint on $\ele$, that is, 
$$Re(X)=\tfrac{1}{2}( (f\otimes g) + (g\otimes f)  ),\, \, \, \, Im(X)=(\tfrac{1}{2i}( (f\otimes g) - (g\otimes f) )$$
and $X=Re(X) + i Im(X)$, with $Re(X)$ and $Im(X)$   $A$-symmetric.  Since $\vp$ is real-valued, we find that
\begin{align*}
\tilde{\vp}_1(X^+) &=\tilde{\vp}_1(Re(X)-iIm(X))=-i\tilde{\vp}_1(iRe(X)) -  \tilde{\vp}_1(iIm(X)) \\
& =-i \vp(iRe(X)) - \vp(i Im(X))= \overline{   i\vp(iRe(X)) -\vp(i Im(X)) }=- \overline{\tilde{\vp}_1 (X)}\, .
\end{align*}
We thus get
\begin{equation}\nonumber 
Tr(Z(g\otimes f))= \tilde{\vp}_1 (g\otimes f)= - \overline{\tilde{\vp}_1 (f\otimes g)}=- \overline{Tr(Z(f\otimes g))}.
\end{equation}
Now it suffices to note that for all $f,g \in \h$,
\[  [Zg,f]=Tr(Z(g\otimes f))= -   \overline{Tr(Z(f\otimes g))} =- [g,Zf],           \]
which proves that $Z$ is $A$-anti-symmetric. To finish the proof of this lemma, we point out that the isomorphism is isometric because $\|\vp \|= \| \vp_1 \|= \| \tilde{\vp}_1\| = \| Z\|_p$.
\end{proof}

 Orbits of finite rank projections in the compatible Grassmannian are related with variational spaces in many-particle Hartree-Fock theory. To briefly show this relationship,  we introduce the following infinite dimensional Stiefel type manifold: for $n \in \NN$, 
\[  \mathcal{C}_n = \{  \, (h_1, h_2 , \ldots , h_n ) \in \h^n   \, : \, [h_i , h_j]=\delta_{i  j}   \, \}. \]
We emphasize that $n$-tuples of vectors in $\mathcal{C}_n$ have  orthonormal coordinates  with respect to the inner product defined by $A$ on the Hilbert space $\ele$. 
Next we consider the following equivalence relation:
\[  (h_1, h_2 , \ldots , h_n ) \sim (g_1, g_2 , \ldots , g_n ) \, \text{ if } \, \sum_{i=1}^n U_{ij} h_i = g_j, \, \, \, j=1, \ldots, n,
\text{ for some $(U_{i j}) \in \mathcal{U}(\CC^n)$}, \] 
where $\mathcal{U}(\CC^n)$ is the unitary group of all $n \times n$ matrices. In the case where $\h=H^1(\RR^3)$ is the first order Sobolev space of $\RR^3$, $\ele=L^2(\RR^3)$ and  $(Af)(x)=\frac{1}{1+|x|^2}f(x)$, the above defined $\mathcal{C}_n$  is usually known as the {\it Stiefel manifold in quantum chemistry} and the quotient space $\mathcal{C}_n / \sim$ is called the {\it Grassmann manifold in quantum chemistry}.

The geometric structure of these Banach manifolds  was studied in \cite{chm}, and does not depend on the specific afore-mentioned function spaces. The main reason for studying geometric properties of these manifolds is to provide a rigorous setting  to apply critical point theory in Hartree-Fock type equations (see \cite{armel}).     In particular, it was shown that  $\mathcal{C}_n / \sim$ is homeomorphic to the orbit $\gg_A \cdot Q_0$ of an  $A$-symmetric projection $Q_0$ of rank $n$. This allows to endow  $\mathcal{C}_n / \sim$   with a manifold structure by making this homeomorphism into a diffeomorphism.

According to Remark \ref{rem connected}, the orbit of an $A$-symmetric projection $Q_0$ of rank $n$ is connected, and it is given by
\[  \gg_A \cdot Q_0 = \{  Q \in \b(\h) \, : \, Q^2=Q, \, Q^*A=AQ, \, \rank(Q)=n \,   \}.   \]
The following result can be seen as a generalization of the fact that the unitary orbit of a finite rank orthogonal projection is a strong symplectic leaf in the Banach Lie-Poisson space of trace class operators.


\begin{coro}\label{symplectic leaves}
For $1\leq p \leq \infty$,  the space $\mathfrak{g}_p$ is a real Banach Lie-Poisson space, and for any $Q_0 \in Gr_A$ such that $rank(Q_0)< \infty$, 
the orbit $\gg_A \cdot Q_0$ is a strong symplectic leaf in  $\mathfrak{g}_p$ endowed with the form
\[  \omega_{Gr_A}(GQ_0G^{-1}) (\, [GXG^{-1},GQ_0G^{-1}]\,, \,[GYG^{-1},GQ_0G^{-1}] \, )= Tr (Q_0[X,Y]), \] 
where $X,Y \in \mathfrak{g}_p$ and  $G \in \gg_{A}$.
\end{coro}
\begin{proof}
According to Remark \ref{rem connected}, any pair of $A$-symmetric projections of the same rank are conjugated by 
an invertible operator $G$ such that $G- 1$ has finite rank.
Then for $1 \leq p \leq \infty$ the orbit may be described as
\[ \gg_A \cdot Q_0 = \{  \, GQ_0G^{-1}\, : \, G  \in \gg_{p,A} \,   \}. \]
Now we are going to verify the hypothesis of \cite[Theorem 7.4]{or}, which in particular will imply that $\mathfrak{g}_p$ is a 
real Banach Lie-Poisson space for $1 \leq p \leq \infty$ and the formula above for the symplectic form. By Proposition \ref{duality}, the predual
of $\mathfrak{g}_q$ is given by  $\mathfrak{g}_p$, where $1 \leq q < \infty$  and $q^{-1} + p^{-1}=1$. Apparently, the coadjoint orbit
satisfies
$  \{ \, GXG^{-1} \, : \, X \in i\mathfrak{g}_p\,  \} \subseteq  i\mathfrak{g}_p  $ for any $G \in \gg_{q,A}$.
Since $Q_0$ has finite rank, it follows that $Q_0 \in i\mathfrak{g}_p$. We will show in Corollary \ref{struc grass comp rest} that 
the projection map induced by the action is a submersion, which means that
$\{ \, G \in \gg_{q,A} \, : \, GQ_0=Q_0G  \,  \}$ is a Banach-Lie subgroup of $\gg_{q,A}$. Again due to the fact that $Q_0$ has finite rank,
it follows that $\gg_q \cdot Q_0 = \gg_p \cdot Q_0$, and by Corollary \ref{struc grass comp rest}, the inclusion map
$\gg_{p,A}
 \cdot Q_0 \hookrightarrow i \mathfrak{g}_p=Q_0 + i \mathfrak{g}_p$ is an injective immersion. By \cite[Theorem 7.5]{or} this latter fact implies that the symplectic form is strong.
\end{proof}


\section{The restricted compatible Grassmannian}\label{res comp gra}

Given a fixed $p$, $1\le p \le \infty$, and a compatible subspace $\s_0\subset \h$, with infinite dimension and co-dimension, with $A$-symmetric idempotent $Q_0=Q_{\s_0}$, we shall define the $p$-restricted  compatible Grassmannian, induced by the direct sum  decomposition 
$$
\h=\s_0 \, \dot{+} \, \n_0,
$$
where $\n_0=\ker(Q_0)$. This decomposition is non orthogonal with respect to the inner product on $\h$.
We shall adopt the following definition, which in the case of the usual restricted Grassmannian, i.e. $A=1$, $\h=\ele$,  is a property equivalent to the definition given, for instance, in \cite{pressleysegal}. 
A compatible subspace $\s$ belongs to the {\it $p$-restricted  compatible Grassmannian} $Gr_{res, p}=Gr_{res, p, Q_0}^A(\h)$ if 
$$
Q_\s-Q_0\in\b_p(\h).
$$
We may think of $Gr_{res,p}$ as the set of all  $A$-symmetric projections satisfying the above equation. We endow 
$Gr_{res,p}$  with the topology defined by $(Q_1,Q_2)\mapsto \|Q_1 - Q_2 \|_p$. Note that
 $Gr_{res,p}$ is smoothly acted on by the group $\gg_{p,A}$. Indeed, if $Q \in Gr_{res, p}$ and $G\in\gg_{p,A}$,
it is clear that $GQG^{-1}\in Gr_A$. Moreover, since $G=1+K$ and $G^{-1}=1+K'$ with $K,K'\in\b_p(\h)$,
$$
GQG^{-1}-Q_0=(1+K)Q(1+K')-Q_0=KQ+QK'+KQK'+(Q-Q_0)\in\b_p(\h).
$$   
As with the usual restricted Grassmannian, the  restricted compatible Grassmannian $Gr_{res,p}$  is not connected. We shall see that the connected components are parametrized by the integers.

\begin{rem}\label{difeo exp} Recall that $\rho_\h$ denotes the spectral radius in $\h$. Note that the exponential map between the open sets (in the $p$-norm)
$$
\{\, X\in\mathfrak{g}_p \, : \, \rho_\h(X)<\pi \, \} \to  \{\, G\in\gg_{p,A} \, : \, -1\notin\sigma_\h(G) \,\}
$$ 
is an analytic diffeomorphism. To prove this assertion, first note that if $X\in \mathfrak{g}_p$ and $\rho_\h(X)<\pi$, then $-1\notin \sigma_\h(e^X)$. Suppose that $e^X=e^Y$ with $X,Y\in\mathfrak{g}_p$ and $X, Y$ as above. According to Theorem \ref{symmetrizable results}, their  extensions $\bar{X}, \bar{Y}$ to $\ele$ verify that $\|\bar{X}\|_{\b(\ele)} = \rho_{\ele}(\bar{X}) \leq \rho_{\h}(X)< \pi$, and $\|\bar{Y}\|_{\b(\ele)}<\pi$. 
On the other hand, $e^{\bar{X}}, e^{\bar{Y}}$ are unitary operators in $\ele$ whose restrictions to $\h$, 
namely $e^X, e^Y$, coincide. It follows that $e^{\bar{X}}=e^{\bar{Y}}$. The bound on the $\b(\ele)$ norms implies that $\bar{X}=\bar{Y}$, and therefore $X=Y$.
The spectrum of $G\in\gg_{p,A}$ lies in the unit circle  and consists  (with the possible exception of $1\in\sigma_\h(G)$)  of eigenvalues with finite multiplicity. Thus if $-1\notin\sigma_\h(G)$, one can define the inverse of the exponential, by means of the usual determination of the logarithm (which is singular in the real negative 
axis, that does not intersect $\sigma_\h(G)$). This map is clearly analytic in $1+\b_p(\h)$. 
Apparently, this logarithm takes values in $\{X\in\mathfrak{g}_p: \rho_\h(X)<\pi\}$.
\end{rem}

Given $Q_1 \in Gr_{res,p}$\,, we denote by $(Gr_{res ,p})_{Q_1}$  the connected component of $Q_1$ in the  restricted compatible Grassmannian.

\begin{prop}\label{comp gras res} 
The following assertions hold:
\begin{enumerate}
\item[i)] The action of $\gg_{p,A}$ on $(Gr_{res ,p})_{Q_1}$ is transitive.
\item[ii)] The map 
$$
\gg_{p,A}\to (Gr_{res ,p})_{Q_1} \subseteq Q_1 + \b_p(\h), \ \ G\mapsto GQ_1G^{-1}
$$
has continuous local cross sections, when both spaces are regarded with the $p$-norm.
\end{enumerate}
\end{prop}
\begin{proof}
\noi $i)$
Since $\gg_{p,A}$ is connected, $\gg_{p,A}\cdot Q_1$ is contained in the component of $Q_1$ in $Gr_{res, p}$\,. We first prove that $\gg_{p,A}\cdot Q_1$ is open and closed in $Gr_{res, p}$\,. To this end, note that $Q_1$ is an interior point of $Gr_{res, p}$\,.  Let $Q \in Gr_{res,p}$ such that $\|Q- Q_1 \|_p< r_{Q_1}$, where $r_{Q_1}$ was defined in Section \ref{diff structure gra}. Since $\|Q- Q_1 \| \leq \|Q-Q_1\|_p <r_{Q_1}$,  
there exists an $A$-anti-symmetric operator  $X_Q$ which is co-diagonal with respect to $Q_1$ and such that $Q=e^{X_Q}Q_1e^{-X_Q}$.  

We claim that $X_Q\in\b_p(\h)$: note that $Q,Q_1\in Gr_{res, p}$ implies that $Q-Q_1\in\b_p(\h)$. Then
 $$
 \epsilon_Q-\epsilon_{Q_1}=2(Q-Q_1)\in\b_p(\h).
 $$
 On the other hand, by construction, 
 $$
 \epsilon_Q-\epsilon_{Q_1}=e^{2X_Q}\epsilon_{Q_1}-\epsilon_{Q_1}=(e^{2X_Q}-1)\epsilon_{Q_1}.
 $$
 It follows that $e^{2X_Q}-1\in\b_p(\h)$, that is, $e^{2X_Q}\in\gg_{p,A}$. Since $\rho_\h(e^{2X_Q}-1)\leq \| e^{2X_Q} - 1 \|< 1$,  then  $-1 \notin \sigma_\h(e^{2X_Q})$. By Remark \ref{difeo exp} we get that $X_Q \in \b_p(\h)$, and our claim is proved. Hence we obtain that $Q=e^{X_Q}Q_1e^{-X_Q}\in \gg_{p, A}\cdot Q_1$.
 
 Let $G_0Q_1G_0^{-1}$ be any other element in this orbit, with $G_0=1+K_0$, $G_0^{-1}=1+K'_0$, and $K_0,K'_0\in\b_p(\h)$. We are going to show that it is also an interior point of the restricted Grassmannian. If $Q\in Gr_{res, p}$ satisfies 
$$
\|Q-G_0Q_1G_0^{-1}\|_p<\frac{r_{Q_1}}{\|G_0\|\|G_0^{-1}\|},
$$
then
$$
\|G_0^{-1}QG_0-Q_1\|_p=\|G_0^{-1}(Q-G_0Q_1G_0^{-1})G_0\|_p\le \|G_0^{-1}\|\|Q-G_0Q_1G_0^{-1}\|_p\|G_0\|<r_{Q_1}.
$$
It follows that $G_0^{-1}QG_0=G^{-1}Q_1G$ for some $G \in \gg_{p,A}$, and we thus get $Q \in \gg_{p,A} \cdot Q_1$. Thus the orbit $\gg_{p, A}\cdot Q_1$ is open in the connected component of $Q_1$ in $Gr_{res, p}$. Note that this assertion is valid for any element $Q$ of this component. Then the complement of $\gg_{p, A}\cdot Q_1$ inside this component, which is a union of disjoint and open orbits, is itself open. Thus $\gg_{p, A} \cdot Q_1$ is also closed in $(Gr_{res, p})_{Q_1}$, and therefore $\gg_{p, A}\cdot Q_1=(Gr_{res, p})_{Q_1}$.

\smallskip

\noi $ii)$ We will only  define a cross section in a neighborhood of $Q_1$\,. A standard translation procedure following the idea of  $i)$ can be 
used to construct a local cross section at any other point of the connected component $(Gr_{res,p})_{Q_1}$\,.

 Let $Q \in Gr_{res,p}$ such that $\| Q- Q_1 \|_p< r_{Q_1}$. The operator $X_Q \in \mathfrak{g}_p$ that we have just defined in $i)$ is actually given by the usual logarithm series
$$
X_Q=\frac12 \, log((2Q_1-1)(2Q-1))=\frac12 \, \sum_{n\ge 1} \frac{(-1)^{n+1}}{n} \, ((2Q_1-1)(2Q-1)-1)^n,
$$
which is convergent in the $p$-norm because $r_{Q_1}<1$. Since the series is the uniform limit of the partial sums, which are continuous functions in the $p$-norm, we can conclude that the map
$$
\{Q\in Gr_{res, p}: \|Q-Q_1\|_p<r_{Q_1}\} \ni Q \mapsto X_Q\in \mathfrak{g}_p
$$
is continuous in the $p$-norm.
\end{proof}

Let us characterize the connected components of $Gr_{res, p}$ in terms of the Fredholm index, as with the usual restricted Grassmannian.  To this effect, several results in the paper \cite{ass} by J. Avron, R. Seiler and B. Simon will be useful. First and foremost, the notion of index for a pair of orthogonal projections. 

Note the following fact. If $Q_1, Q_2\in Gr_{res, p}$, their self-adjoints extensions $\bar{Q}_1$ and $\bar{Q}_2$  verify that $\bar{Q}_1-\bar{Q}_2\in\b_p(\ele)$. Indeed, by Theorem \ref{symmetrizable results}, $\bar{Q}_1-\bar{Q}_2$ 
are self-adjoint compact operators, whose eigenvalues coincide with the eigenvalues of $Q_1-Q_2$. Moreover, the absolute values of these eigenvalues, 
by a classical result \cite{lalesco}, are bounded by the singular values of $Q_1-Q_2$. It follows that $\bar{Q}_1-\bar{Q}_2\in\b_p(\ele)$. Therefore, $\bar{Q}_1,\bar{Q}_2$ belong to the classical restricted Grassmannian in the Hilbert space $\ele$, given by the orthogonal polarization
$$
\ele= R(\bar{Q}_0)\oplus N(\bar{Q}_0).
$$
In particular, this implies that 
$$
\bar{Q}_2\bar{Q}_1|_{R(\bar{Q}_1)}:R(\bar{Q}_1)\to R(\bar{Q}_2)
$$
are Fredholm operators. According to \cite{ass},  the index of this operator  is the index $index(\bar{Q}_1,\bar{Q}_2)$ of the pair $\bar{Q}_1,\bar{Q}_2$.
\begin{teo}\label{components gres}
Let $Q_1,Q_2\in Gr_{res, p}$. Then $Q_1$ and $Q_2$ belong to the same connected component if and only if $index(\bar{Q}_1,\bar{Q}_0)=index(\bar{Q}_2,\bar{Q}_0)$.
\end{teo}
\begin{proof}
Let $Q(t)$, $t\in[t_1,t_2]$, be a continuous path in $Gr_{res, p}$ such that $Q(t_1)=Q_1$ and $Q(t_2)=Q_2$. The inequality 
$$
\|\bar{Q}(t)-\bar{Q}(s)\|_{\b(\ele)} \le \|Q(t)-Q(s)\|,
$$
implies that the path $\bar{Q}(t)$ is continuous in the restricted Grassmannian of $\ele$. This implies that $index(\bar{Q}_1,\bar{Q}_0)=index(\bar{Q}_2,\bar{Q}_0)$. 

Conversely, suppose that $index(\bar{Q}_1,\bar{Q}_0)=index(\bar{Q}_2,\bar{Q}_0)$.
By Theorem 3.4 in \cite{ass}, this implies that
$$
index(\bar{Q}_1,\bar{Q}_2)=index(\bar{Q}_1,\bar{Q}_0)+index(\bar{Q}_0,\bar{Q}_2)=index(\bar{Q}_1,\bar{Q}_0)-index(\bar{Q}_2,\bar{Q}_0)=0.
$$
Since $Q_1-Q_2$ is compact, by the results of M.G. Krein and P. D. Lax, the eigenspaces corresponding to the non zero eigenvalues of $Q_1-Q_2$ and $\bar{Q}_1-\bar{Q}_2$ coincide. In particular, this implies that  
$$
\ker(\bar{Q}_1-\bar{Q}_2-1)=\ker(Q_1-Q_2-1) \hbox{ and } \ker(\bar{Q}_1-\bar{Q}_2+1)=\ker(Q_1-Q_2+1)
$$ 
are finite dimensional subspaces of $\h$. As remarked in \cite{ass}, the hypothesis $index(\bar{Q}_1,\bar{Q}_2)=0$ implies that 
$$
\dim \ker(\bar{Q}_1-\bar{Q}_2-1)=\dim \ker(\bar{Q}_1-\bar{Q}_2+1).
$$
Let $W_0$ be a unitary transformation (for the inner product $[\ , \ ]$ of $\ele$) from $\ker(\bar{Q}_1-\bar{Q}_2-1)$ onto $\ker(\bar{Q}_1-\bar{Q}_2+1)$.  
Let ${\cal N}$ be the subspace 
$$
{\cal N}=\ker(\bar{Q}_1-\bar{Q}_2-1)\oplus \ker(\bar{Q}_1-\bar{Q}_2+1)\subset\h,
$$ 
and $U_0:{\cal N}\to {\cal N}$  be the unitary operator
$$
U_0(f\oplus g)=W_0^{-1}g\oplus W_0f.
$$
Since ${\cal N}$ is finite dimensional, it  has  a supplement $\h_0$ which is orthogonal for the inner product $[\ , \ ]$ of $\ele$ (see Proposition \ref{caract comp}): 
$$
{\cal N}\oplus \h_0=\h  \hbox{ and } {\cal N}\oplus \bar{\h}_0= \ele.
$$
Straightforward computations show that 
$$
Q_1(\ker(Q_1-Q_2-1))\subset {\cal N} \hbox{ and } Q_1(\ker(Q_1-Q_2+1))\subset \ker(Q_1-Q_2-1),
$$
and therefore $Q_1({\cal N})\subset {\cal N}$. Similarly, $Q_2({\cal N})\subset {\cal N}$. Therefore $Q_1$ and $Q_2$ leave $\h_0$ invariant. To complete the proof, we have to construct an invertible operator $G_0$ on $\h_0$, 
intertwining $Q_1|_{\h_0}$ and $Q_2|_{\h_0}$, such that $G_0$ is isometric for the inner product $[\ , \ ]$ of $\ele$ and belongs to $1+\b_p(\h_0)$. Clearly, this would imply that $U_0 \oplus G_0$ belongs to $\gg_{p,A}$ and intertwines $Q_1$ and $Q_2$. 

Let $B=1-Q_1-Q_2$. 

Clearly $B$ is reduced by the decomposition ${\cal N}\oplus \h_0=\h$. The extension $\bar{B}$ of $B$ to $\ele$ is invertible (see \cite{ass}). Indeed, since 
$B^2=1-(Q_1-Q_2)^2$, the spectrum of $\bar{B}^2$ consists of $1$ and eigenvalues of finite multiplicity, corresponding to the squares of the non zero eigenvalues of $Q_1-Q_2$. Since the eigenvalues $1$ and $-1$ have been erased in $\h_0$, it follows that $B$ 
is invertible in $\h_0$. Set 
$$
S=Q_2Q_1+(1-Q_2)(1-Q_1).
$$
Note that $S$ also is reduced by $\h_0\oplus {\cal N}=\h$. Denote by $S_0=S|_{\h_0}$. It is easily seen that 
$$
S=(1-2Q_2)B=B(1-2Q_1),
$$
which implies that $S_0$ is invertible in $\h_0$.
Moreover,
$$
S=1+2Q_2Q_1-Q_1-Q_2=1+Q_2(Q_1-Q_2)+(Q_2-Q_1)Q_1 \in 1+\b_p(\h),
$$
and  then $S_0\in 1_{\h_0}+\b_p(\h_0)$. Let us obtain from $S_0$ an $\ele$-isometry  in $1_{\h_0}+\b_p(\h_0)$  , intertwining $Q_1$ and $Q_2$.  
Let $S'_0=Q_1Q_2+(1-Q_1)(1-Q_2)$  acting on $\h_0$. Clearly it is invertible, and intertwines $Q_2$ with $Q_1$. Moreover, $\bar{S}'_0=\bar{S}_0^+$ (as in Theorem \ref{proper op}). Put 
$$
R=S'_0S_0=Q_1Q_2Q_1+(1-Q_1)(1-Q_2)(1-Q_1),
$$
which is invertible and commutes with $Q_1$. Moreover, it is symmetrizable on $\h_0$, that is, $R^*P_{\h_0}AP_{\h_0}=P_{\h_0}AP_{\h_0}R$. Also note that the extension $\bar{R}=|\bar{S}_0|^2$ is positive and invertible on $\bar{\h}_0$. The spectrum of $R$ consists of finite multiplicity positive eigenvalues and the scalar $1$. The element $R$ is of the form $1+\b_p(\h_0)$, therefore it is an invertible element in the Banach algebra $\mathbb{C}+\b_p(\h_0)$ (endowed with the norm $|z1+K|=(|z|^p+\|K\|_p^p)^{1/p}$). Let $\Gamma$ be a path in $\mathbb{C}$, which contains the spectrum $\sigma_{\h_0}(R)$ in its interior,  leaves $0$ outside, and is symmetric with respect to the $x$-axis. Consider the invertible element $T\in\b(\h_0)$ given by the Riesz integral
$$
T=\frac{1}{2\pi i} \int_\Gamma e^{\frac12 log(z)}(z1-R)^{-1} dz,
$$
where $log(z)$ is the usual determination of the complex logarithm. Then
$$
T^*P_{\h_0}AP_{\h_0}=-\frac{1}{2\pi i} \int_{\bar{\Gamma}} e^{\frac12 log(\bar{z})}(\bar{z}1-R^*)^{-1}P_{\h_0}AP_{\h_0} \ d\bar{z}.
$$
Since $R$ is symmetrizable,  for any integer power $k\ge 0$,
$$
(R^*)^kP_{\h_0}AP_{\h_0}=P_{\h_0}AP_{\h_0}R^k.
$$
Therefore
$$
(\bar{z}1-R^*)^{-1}P_{\h_0}AP_{\h_0}=P_{\h_0}AP_{\h_0}(\bar{z}1-R)^{-1},
$$
and thus
$$
T^*P_{\h_0}AP_{\h_0}=-P_{\h_0}AP_{\h_0}\frac{1}{2\pi i} \int_{\bar{\Gamma}} e^{\frac12 log(\bar{z})}(\bar{z}1-R)^{-1} d\bar{z}.
$$
Since $\Gamma$ is symmetric with respect to the $x$-axis, if $z(t)$, $t\in I$ is a counter-clockwise parametrization of $\Gamma$, then  $\bar{z}(t)$, $t\in I$ is a clockwise reparametrization of the same path. Thus,
$$
T^*P_{\h_0}AP_{\h_0}=P_{\h_0}AP_{\h_0}T,
$$
i.e. $T$ is a symmetrizable operator in $\h_0$, and it is of the form $w1+\b_p(\h_0)$. Also it is apparent that $T$ commutes with $Q_1$, because $R$ does. The  extension of $T$  to $\ele$ is  the square root of $\bar{R}$, and therefore $w=1$, i.e. $R\in 1+\b_p(\h_0)$. Then $G_0=S_0T^{-1}$ is an invertible element acting in $\h_0$, 
which induces an invertible element in  $\bar{\h}_0$, given by 
$$
\bar{G}=\bar{S}_0\bar{T}^{-1}=\bar{S}_0|\bar{S}_0|^{-1},
$$
which is a unitary operator in $\bar{\h}_0$. Clearly, $G_0$ intertwines $Q_1|_{\h_0}$ and $Q_2|_{\h_0}$, and the proof is complete.
\end{proof}

There is an analogue of Proposition \ref{delta splits} in this context, which will allow us to prove the local regularity of the component $(Gr_{res ,p})_{Q_1}$. Recall that  $\delta_{Q_1}$ is defined by $\delta_{Q_1}(X)=XQ_1-Q_1X$. Note that if $X\in\mathfrak{g}_p$, then $\delta_{Q_1}(X)\in i \mathfrak{g}_p$. Indeed, it was already seen that it belongs to $\b_s^A(\h)$, and it is apparent that $\delta_{Q_1}(X)\in\b_p(\h)$, if $X\in\b_p(\h)$.
\begin{lem}
The map
$$
\delta_{Q_1}|_{\mathfrak{g}_p}:\mathfrak{g}_p\to i \mathfrak{g}_p, \ \ \delta_{Q_1}(X)=XQ_1-Q_1X
$$
has complemented range and kernel.
\end{lem}
\begin{proof}
Note that by a similar argument as above, $\delta_{Q_1}(i \mathfrak{g}_p)\subset \mathfrak{g}_p$. To avoid confusion, let us denote by $\delta_1$ the map from $\mathfrak{g}_p$ to $i \mathfrak{g}_p$, and by $\delta_2$ the map from $i \mathfrak{g}_p$ to $\mathfrak{g}_p$, both induced by restricting $\delta_{Q_1}$. Note that 
$$
\delta_1\delta_2\delta_1=\delta_1 \ \hbox{ and } \ \ \delta_2\delta_1\delta_2=\delta_2.
$$
Indeed, if $X\in \b(\h)$, an elementary computation shows that
$$
\delta_{Q_1}^3(X)=XQ_1-Q_1X=\delta_{Q_1}(X),
$$
and the claim follows. These formulas imply that $\delta_2$ is a pseudo-inverse for $\delta_1$. Therefore $\delta_2\delta_1$ is an idempotent operator acting in $\mathfrak{g}_p$, whose kernel clearly contains the kernel of $\delta_1$. On the other hand, if $X$ lies in the kernel of $\delta_2\delta_1$, then apparently $0=\delta_1\delta_2\delta_1(X)=\delta_1(X)$, i.e. $\ker(\delta_2\delta_1)=\ker(\delta_1)$, and therefore $\ker(\delta_1)$ is complemented in $\mathfrak{g}_p$.

Similarly, $\delta_1\delta_2$ is an idempotent operator in $i \mathfrak{g}_p$, whose range is contained in the range of $\delta_2$. If $Y=\delta_1(X)$ for some $X\in \mathfrak{g}_p$, then 
$$
\delta_1\delta_2(Y)=\delta_1\delta_2\delta_1(X)=\delta_1(X)=Y,
$$
i.e. $R(\delta_1\delta_2)=R(\delta_1)$, and therefore it is complemented in $i \mathfrak{g}_p$.
\end{proof}
Therefore Lemma \ref{raeburn} applies:
\begin{coro}\label{struc grass comp rest}
The component $(Gr_{res ,p})_{Q_1}$ of the compatible restricted Grassmannian is a complemented submanifold of $Q_0+i \mathfrak{g}_p$ and the map
$$
\gg_{A, p}\to (Gr_{res ,p})_{Q_1}, \ \ G\mapsto GQ_1G^{-1}
$$
is a $C^\infty$ submersion.
\end{coro}

There is yet another consequence of the fact that the logarithm $Q\mapsto X_Q$ takes values in $\mathfrak{g}_p$ if $\|Q-Q_1\|_p<r_{Q_1}$:
\begin{coro}\label{geodesica restringida}
Let $Q$ be an element in the component of $Q_1$ in $Gr_{res, p}$. If $\|Q-Q_1\|_p<r_{Q_1}$, then the unique geodesic $\delta(t)=e^{tX_{Q}}Q_1e^{-tX_{Q}}$ of the full compatible Grassmannian $Gr_A$ which joins $Q$ and $Q_1$, lies in fact inside the restricted Grassmannian $Gr_{res, p}$.
\end{coro}

\section{Finsler metrics in $Gr_A$}\label{minimal curves}

We shall endow the tangent spaces of $Gr_A$  with a continuous Finsler metric.  Recall first that the tangent space of $Gr_A$ at $Q_0$ is given by
$$(TGr_A)_{Q_0}=\{XQ_0-Q_0X: X\in\b_{as}^A(\h)\}.$$ 
Note that the elements of the group $\gg_A$ are isometries for the $\ele$ inner product. Moreover, the action is isometric if one chooses for  $(TGr_A)_{Q}$ the norm of $\b(\ele)$. Namely, for $v\in (TGr_A)_{Q}$, put
$$
|v|_Q=\|\bar{v}\|_{\b(\ele)}.
$$
Then if $v=XQ_0-Q_0X\in (TGr_A)_{Q_0}$ and $G\in\gg_A$
$$
|G\cdot v|_{G\cdot Q}=\|G(XQ-QX)G^{-1}\|_{\b(\ele)}=\|XQ-QX\|_{\b(\ele)}=|v|_Q.
$$
We denote the length functional induced by this metric with
\[   L_{Gr_A}(\gamma)=\int_0^1 \| \dot{\gamma} \|_{\b(\ele)}  \, dt , \]
where $\gamma: [0,1] \to Gr_A$ is a piecewise smooth curve.
The rectifiable distance associated to this metric is defined as usual, i.e.
\[  d(Q_0,Q_1)=\inf \{ \, L_{Gr_A}(\gamma) \, : \, \text{$\gamma$ is a piecewise smooth curve in $Gr_A$ joining $Q_0$ and $Q_1$}   \, \}. 
 \]
 
\begin{rem}\label{complt}
It is easily  seen  that this rectifiable distance defines a  topology weaker than the operator norm of $\b(\h)$. 
On the other hand, the metric space $(Gr_A, d)$ is not complete. In order to show this latter fact, pick a vector $f \in \ele \backslash \h$ such that $\| f\|_A=1$. Then there exists a sequence of vectors  $(f_n)_n$ in $\h$ such that $\|f\|_A=1$ and $\| f_n - f \|_A \to 0$. Next consider the rank one operators defined by $(f_n\otimes f_n)(h)=[h,f_n]f_n$, $h \in \h$. Clearly, 
$f_n \otimes f_n \in Gr_A$.

Given $n,m \geq 1$, we can find an operator $G_{n,m} \in \gg_A$ such that $G_{n,m} (f_n \otimes f_n) G_{n,m}^{-1}=f_m \otimes f_m$. Indeed, $G_{n,m}$ can be chosen satisfying $\rank (G_{n,m} -1) \leq 2$ (see \cite[Lemma 3.4]{chm}). According to \cite[Proposition 4.5]{achl} the operators $G_{n,m}$ have  logarithms $X_{n,m}$ of finite rank in $\b_{as}^A(\h)$.   Consider the curves $\gamma_{n,m}(t)=e^{tX_{n,m}} (f_n \otimes f_n ) e^{-tX_{n,m}}$, then
\begin{align*}
  d(f_n \otimes f_n,f_m \otimes f_m) & \leq L_{Gr_A}(\gamma_{n,m})   \\ 
  &  \leq \| X_{n,m}(f_n \otimes f_n) - (f_n \otimes f_n)   X_{n,m} \|_{\b(\ele)} \leq 2 \| X_{n,m} \|_{\b(\ele)} \underset{n,m \to \infty}{\longrightarrow} 0,
\end{align*}
where this convergence to zero is due to the well known fact that orthogonal projections in $\b(\ele)$ have local continuous cross sections    
and $\| f_n \otimes f_n - f \otimes f \|_{\b(\ele)} \to 0.$ Hence we have proved that $(f_n \otimes f_n)_n$ is a Cauchy sequence in $(Gr_A, d)$.

Now we will show that $(f_n \otimes f_n)_n$ does not converge in $(Gr_A, d)$. Suppose that there is $Q_0 \in Gr_A$ such that $d(f_n \otimes f_n,Q_0) \to 0$. Since straight lines are shortest paths in any vector space, 
\[  \| \bar{Q}_0 - f_n \otimes f_n \|_{\b(\ele)} \leq \int_0 ^1 \| \bar{\dot{\gamma}} \|_{\b(\ele)} dt  \]
for any piecewise curve $\gamma: [0,1] \to Gr_A$ joining $\bar{Q}_0$ and $f_n \otimes f_n$. Therefore
\[   \| \bar{Q}_0 - f_n \otimes f_n \|_{\b(\ele)} \leq d(Q_0 , f_n \otimes f_n) \to 0.  \]
Thus we get that $\bar{Q}_0=f \otimes f$. By our choice of the vector $f$, it follows that $\bar{Q}_0$ does not leave $\h$ invariant. This  contradicts our assumption that $Q_0 \in Gr_A$. 
\end{rem}

With this choice for the metric one can prove that the geodesics in Corollary \ref{geodesicas} are curves of minimal length in $Gr_A$. 
\begin{prop}\label{minimal}
Let $Q_0,Q_1\in Gr_A$ such that $\|Q_0-Q_1\|<r_{Q_0}$. Then the unique geodesic of ${\cal Q}(\h)$ which joins them, and lies inside $Gr_A$, has minimal length among all possible piecewise smooth curves in $Gr_A$ joining  $Q_0$ and $Q_1$.
\end{prop}
\begin{proof}
 Let $\gamma(t)$, $t\in I$, be a piecewise smooth curve in $Gr_A$. Then $\bar{\gamma}(t)$, which is a curve in the set of orthogonal projections in $\b(\ele)$, is also piecewise smooth. Indeed, the extension map
$$
\b_{s}^A(\h) \to \b(\ele), \, \, \, \,    X \mapsto \bar{X} 
$$
 is contractive (see Remark \ref{ext map bounded}). Note that by Corollary \ref{geodesicas}, the condition $\|Q_0-Q_1\|<r_{Q_0}$ implies the existence of the geodesic $\delta(t)=e^{tZ_{Q_0}}Q_0e^{-tZ_{Q_0}}$ joining $Q_0$ and $Q_1$. The extension  $\bar{\delta}$ of this curve is given by
$$
\bar{\delta}(t)=e^{t\bar{Z}_{Q_0}}\bar{Q}_0e^{-t\bar{Z}_{Q_0}}=e^{tZ_{\bar{Q}_0}}\bar{Q}_0e^{-tZ_{\bar{Q}_0}},
$$
which is the unique geodesic joining $\bar{Q}_0$ and $\bar{Q}_1$ in the manifold of symmetric projections in $\b(\ele)$, and $\|\bar{Z}_{Q_0}\|_{\b(\ele)}\le \pi/6 <\pi/2$. Then by the result in \cite{pr}, 
$$
length(\bar{\delta})\le length(\bar{\gamma}),
$$
where the lengths are measured with the usual operator norm of $\b(\ele)$.
On the other hand, by the definition of the metric in $Gr_A$, it follows that
$$
L_{Gr_A}(\gamma) =\int_I \|\bar{\dot{\gamma}}(t)\|_{\b(\ele)} dt =\int_I \|\dot{\bar{\gamma}}(t)\|_{\b(\ele)}dt=length(\bar{\gamma}).
$$ 
and the same holds true for $\delta$. This completes the proof.
\end{proof}
A natural Finsler metric can also be defined in the restricted compatible Grassmannian $Gr_{res, p}$. As seen in the last section, the connected components are homogeneous spaces of the Banach-Lie group $\gg_{A, p}$. Let us introduce a metric which is invariant for the action of this group. Namely, a  tangent  vector $v$ to the component of $Q_1$ of $Gr_{res, p}$ at $Q$, is the velocity vector $v=\dot{\alpha}(0)$ of a curve which takes values in the $A$-symmetric part of $Q_1+\b_p(\h)$. Therefore tangent vectors are $A$-symmetric operators in $\b_p(\h)$. Then put
$$
|v|_{Q}=\|\bar{v}\|_p,
$$
the Schatten $p$-norm of the extension $\bar{v}\in\b_p(\ele)$. Associated with this metric there is a non complete rectifiable distance (same proof as Remark \ref{complt}). Elements $G$ in $\gg_{A,p}$ extend to unitary operators $\bar{G}$ in $\ele$, which due to the results of M. G. Krein and P. D. Lax, are of the form $1+\b_p(\ele)$. Let $\gamma(t)$ be a smooth curve in $Gr_{res, p}$ with $\gamma(0)=Q_1$. Then $\bar{\gamma}$ is a curve in the restricted Grassmannian of $\ele$ given by the polarization $\ele=R(\bar{Q}_0)\oplus N(\bar{Q}_0)$. If $\|Q-Q_1\|_p<r_{Q_1}$, by Corollary \ref{geodesica restringida}, there exists a unique geodesic $\delta(t)=e^{tX_Q}Q_1e^{-tX_Q}$ such that $\delta(0)=1$. The assumption $r_{Q_1}<1$, implies in particular, that 
$\|\bar{Q}-\bar{Q}_1\|_{\b(\ele)}\le \|Q-Q_1\|_p<1$. Therefore by  \cite[Theorem 5.5]{odospe}, the geodesic $\bar{\delta}$ has minimal length in the restricted Grassmannian of $\ele$. Thus we have the following:
\begin{coro}\label{minimalp}
Let $Q,Q_1\in Gr_{res, p}$ such that $\|Q-Q_1\|_p<r_{Q_1}$. Then the unique geodesic $\delta$ of the compatible Grassmannian which joins them (and which remains inside $Gr_{res, p}$), has minimal length for the Finsler metric defined above.
\end{coro}

\medskip
\medskip

{\sc (Esteban Andruchow)} {Instituto de Ciencias,  Universidad Nacional de Gral. Sar\-miento,
J.M. Gutierrez 1150,  (1613) Los Polvorines, Argentina and Instituto Argentino de Matem\'atica, `Alberto P. Calder\'on', CONICET, Saavedra 15 3er. piso,
(1083) Buenos Aires, Argentina.}

\noi e-mail: {\sf eandruch@ungs.edu.ar}

\bigskip

{\sc (Eduardo Chiumiento)} {Departamento de de Matem\'atica, FCE-UNLP,Calles 50 y 115, 
(1900) La Plata, Argentina  and Instituto Argentino de Matem\'atica, `Alberto P. Calder\'on', CONICET, Saavedra 15 3er. piso,
(1083) Buenos Aires, Argentina.}     
                                               
\noi e-mail: {\sf eduardo@mate.unlp.edu.ar}                                       

\bigskip

{\sc (Mar\'ia Eugenia Di Iorio y Lucero)} {Instituto de Ciencias,  Universidad Nacional de Gral. Sarmiento,
J.M. Gutierrez 1150,  (1613) Los Polvorines, Argentina and Instituto Argentino de Matem\'atica, `Alberto P. Calder\'on', CONICET, Saavedra 15 3er. piso,
(1083) Buenos Aires, Argentina.}

\noi e-mail: {\sf mdiiorio@ungs.edu.ar }
\end{document}